\documentclass[10pt,reqno]{article}

\usepackage{amsmath}
\usepackage{latexsym}
\usepackage{amsmath,amsthm}
\usepackage{amsfonts}
\usepackage{amssymb}
\usepackage{amsfonts,euscript}
\usepackage{graphicx}

\DeclareGraphicsExtensions{.eps,epsfig,.bmp,.jpg,.pdf,.mps,.png,.gif}
\newtheorem{theorem}{\bf Theorem}[section]

\newtheorem{proposition}{\bf Proposition}[section]

\newtheorem{lemma}{\bf Lemma}[section]
\newtheorem{definition}{\bf Definition}[section]
\newtheorem{example}{\bf Example}[section]

\newtheorem{remark}{\bf Remark}[section]

\newcommand{\beq}{\begin{equation}}
\newcommand{\eeq}{\end{equation}}
\newcommand{\beqn}{\begin{eqnarray}}
\newcommand{\eeqn}{\end{eqnarray}}
\newcommand{\bear}{\begin{array}}
\newcommand{\eear}{\end{array}}
\newcommand{\beit}{\begin{itemize}}
\newcommand{\eeit}{\end{itemize}}
\newcommand{\beqno}{\begin{eqnarray*}}
\newcommand{\eeqno}{\end{eqnarray*}}

\renewcommand{\b}{\beta}

\allowdisplaybreaks

\title{The total variation flow perturbed by gradient linear multiplicative noise}
\date{}

\numberwithin{equation}{section}

\begin{document}
\maketitle

\centerline{\scshape Michael R$\ddot{o}$ckner}
\medskip
{\footnotesize
 \centerline{Fakultat fur Mathematik, Universitat Bielefeld  
}
     \centerline{D-33501 Bielefeld, Germany}
\centerline{e-mail: roeckner@math.uni-bielefeld.de}} 

\medskip

 \centerline{\scshape Ionu\c t Munteanu}
\medskip
{\footnotesize
  \centerline{Alexandru Ioan Cuza University of Ia\c si, Department of Mathematics }
   \centerline{Blvd. Carol I, no.11, 700506-Ia\c si, Romania}

   }
   {\footnotesize
  \centerline{Octav Mayer Institute of Mathematics, Romanian Academy }
   \centerline{Blvd. Carol I, no.8, 700505-Ia\c si, Romania}
\centerline{e-mail: ionut.munteanu@uaic.ro}
   }

\medskip
\noindent \textbf{Abstract}.We consider stochastic non-linear diffusion equations with a highly singular diffusivity term and multiplicative  gradient-type noise. We study existence and uniqueness  of non-negative variational solutions in terms of stochastic variational inequalities. We also show extinction in finite time with probability one. These kind of equations arise, e.g. in the use for simulation of image restoring techniques or for modelling turbulence.
\section{Introduction of the model}
We are concerned here with equations of the form
\begin{equation}\label{ioi20}\left\{\begin{array}{l}dX(t)=\text{div}[\text{sgn}(\nabla X(t))]dt+\frac{1}{2}\text{div}[\textbf{b}^\mathcal{T}\textbf{b}\nabla X(t))]dt+\left<\textbf{b}\nabla X(t),d\beta(t)\right>_{\mathbb{R}^N}\text{ in }(0,T)\times \mathcal{O},\\
X(t)=0 \text{ on }(0,T)\times \partial\mathcal{O},\\
X(0)=x \text{ in }\mathcal{O}.\end{array}\right.\ \end{equation}Here, $\mathcal{O}\subset \mathbb{R}^d,\ d\in\mathbb{N}^*$, is a bounded open domain with the boundary $\partial\mathcal{O}$ of class $C^3$; the multivalued function $\text{sgn}$ is given by
$$\text{sgn}(x)=\left\{\begin{array}{cc}\frac{x}{|x|}, & \text{ for } x\neq 0,\\
\left\{\xi\in\mathbb{R}^d:\ |\xi|\leq 1\right\}, & \text{ for } x=0.
\end{array}\right.\ $$ Further,  $N\in\mathbb{N}^*,\ b_i:\mathbb{R}^d\rightarrow \mathbb{R}^d,\ 1\leq i\leq N,$ and 
$$\textbf{b}=\left(\begin{array}{c}b_1\\b_2\\...\\b_N\end{array}\right)\in \mathbb{R}^{N\times d};$$ and $\beta=(\beta_1,\beta_2,...,\beta_N)$ denotes an $N-$dimensional Brownian motion on a filtered probability space $(\Omega,\mathcal{F},\ \left\{\mathcal{F}_t\right\}_{t\geq0},\mathbb{P})$. Here, $\textbf{b}^\mathcal{T}$ stands for the transpose of the matrix $\textbf{b}$. Finally, the initial data $x\in L^2(\mathcal{O})$.

To illustrate this problem let us consider the following  partial differential equation
\begin{equation}\label{e21}\partial_t X(t)=\text{div}[\text{sgn}(\nabla X(t))]+v\cdot\nabla X(t),\ \text{ in } [0,\infty)\times \mathcal{O},\end{equation}which arises, e.g., in material science, see \cite{1d}.   The function $X$ can be interpreted as a density of a substance diffusing in a continuum, moving with a velocity $v$.  When turbulence occurs it is difficult to   determine $v$ precisely, so, one should consider the random velocity field $v$:
$$v(t,\xi)=\sum_{i=1}^N b_i(\xi)\frac{d\beta_i(t)}{dt}.$$ (For further details, see \cite{brez} and the references therein). Plugging this velocity into  (\ref{e21}), we arrive to the following Stratonovich equation
\begin{equation}\label{e22}d X(t)=\text{div}[\text{sgn}(\nabla X(t))]dt+\sum_{i=1}^N\left[b_i(t)\cdot\nabla X(t)\right]\circ d\beta_i(t),\ \text{ in } [0,\infty)\times \mathcal{O},\end{equation}that is our It\^o equation (\ref{ioi20}). So, for modelling turbulence in the flux of a diffusing material, one should perturb the continuity equation by a gradient Stratonovich noise, as above ( see \cite{1a,1b,1c}). Similar kind of equations as (\ref{ioi20}), with multiplicative gradient-type noise, have been considered for example in \cite{m}, for modelling turbulence in the Navier-Stokes equations, or in \cite{mag}, for the Magnetohydrodynamic equations. 

Besides this, such equations arise in image processing techniques in \cite{six, wang}, where the authors show  that considering gradient dependent noise, the numerical simulation results prove that the solution of this model improve the solution obtained by the TV regularization. Other examples, and moreover, further details on the complexity of the present subject can be found in \cite{bok}. Finally, it should be emphasized that this paper solves an open problem addressed in \cite{br1,br2}.

Due to its high singularity, equation (\ref{ioi20}) does not have a solution in the standard sense for all $L^2(\mathcal{O})-$initial solutions, i.e., as an It\^o integral equation. That is why, we shall reformulate  it in the framework of stochastic variational inequalities (see Definition \ref{d1} below). In this paper, we prove the existence and uniqueness of variational solutions to (\ref{ioi20}) (see Theorem \ref{t1} below). In the literature, there are some results of this type for similar models, namely,  for the  non-linear diffusion equation
$$\partial_t X(t)=\text{div}[\text{sgn}(\nabla X(t))],$$ perturbed by an additive  continuous noise $dW(t)$, in \cite{b4}; and perturbed by a multiplicative noise $X(t)d W(t)$, in \cite{b3}. A recent preprint  \cite{ciotir} is dealing with a similar equation as (\ref{ioi20}), but with Neumann boundary conditions, whereas we consider Dirichlet boundary conditions. Also, their approach is different from ours. However, we contacted the authors and they say that the paper is still  under revision, since there are  some issues that need to be solved.

To achieve our goal, we further develop the ideas in \cite{b3}. But, there are some important differences, since, unlike \cite{b3}, here we have a gradient-type multiplicative noise. We approximate equation (\ref{ioi20}) by equation (\ref{e}) below (namely, we replace the multi-valued function $sgn$ with its Yosida approximation), and show the existence and uniqueness for it. To this end, by scaling $Y_\lambda=e^{-\sum_i \beta_i B_i}X_\lambda$, we rewrite it equivalently as the random  deterministic equation   (\ref{ee9}). As mentioned in \cite{b3}, this equivalent reformulation of (\ref{e}) is crucial for the uniqueness part. Besides this, in the present case, it turns out that it is also crucial for obtaining the mandatory $H^2(\mathcal{O})$-regularity of the approximation solution. Roughly speaking, when trying to prove such a strong regularity, one must assume a commutativity between the operators $B_i,\ i=1,...,N$ (introduced in (\ref{momo1}) below) and the resolvent of the Dirichlet Lapalcian  $J_\epsilon$ (introduced in (\ref{Ae}) below), i.e., $J_\epsilon B_i=B_i J_\epsilon$ for all $i=1,...,N$. This kind of hypothesis has been, e.g.,  employed in \cite{ciotir}. But, this  is unlikely to hold because of the difference between the ranges of the operators involved (more precisely, $B_iu\in L^2(\mathcal{O})$, while $J_\epsilon u\in H^2(\mathcal{O})\cap H_0^1(\mathcal{O}),\ u\in H_0^1(\mathcal{O})$). Instead, we assume that the group generated by $B_i$  (defined in (B3) below) commutes with $J_\epsilon$, that is exactly hypothesis $\mathbf{(H_\Delta)}$ below. This is more natural due to the fact that $e^{sB_i}$ preserves $H^2(\mathcal{O})\cap H_0^1(\mathcal{O}).$ This leads to the next approach: firstly to show the $H^2-$ regularity of the scaled variable $Y_\lambda$, then, after showing the equivalence, deduce the $H^2-$regularity for $X_\lambda$. In Example \ref{ex} below we give nontrivial examples of $b_i,\ i=1,...,N$ such that hypothesis $\mathbf{(H_\Delta)}$ holds. We stress that, in our case,  our results are stronger than the corresponding ones in \cite{b3}, in the situation considered there, because, here we obtain pathwise existence and uniqueness (see Definition \ref{d1} and Theorem \ref{t1} below). This is a consequence of the fact that the It\^o's formula for the $L^2-$norm of the solution of the approximation equation (\ref{e}), $X_\lambda$, does not contain a stochastic part (due to the skew-adjointness of the operators $B_i,\ i=1,...,N,$ see (\ref{lele1}) below) and the uniform pathwise convergence of $X_\lambda$ in (\ref{e8e}) below. Besides this, here we obtain  the extinction in finite time of the solutions with probability one  (stronger than in \cite{b3}, where the authors prove this only with positive probability), see Theorem \ref{la17} and \ref{Tit3} below.  Finally, we also prove a result concerning the positivity of the solutions, see  Theorem \ref{Tit1} below.

\section{Preliminaries}
For every $1\leq p\leq\infty$, by $L^p(\mathcal{O})$, we denote the space of all Lebesgue $p-$integrable functions on $\mathcal{O}$ with the norm $|\cdot|_p$. The scalar product in $L^2(\mathcal{O})$ is denoted by $\left<\cdot,\cdot\right>$.  $W^{1,p}(\mathcal{O})$ denotes the standard Sobolev space $\left\{u\in L^p(\mathcal{O});\ \nabla u\in L^p(\mathcal{O})\right\}$ with the corresponding norm
$$\|u\|_{1,p}:=\left(\int_\mathcal{O}|\nabla u|^pd\xi\right)^\frac{1}{p}+|u|_p,$$where $d\xi$ denotes the Lebesgue measure on $\mathcal{O}$. $W_0^{1,p}(\mathcal{O})$ denotes the space $\left\{u\in W^{1,p}(\mathcal{O});\ u= 0\text{ on }\partial\mathcal{O}\right\}.$ We set $H_0^1(\mathcal{O})=W^{1,2}_0(\mathcal{O})$, $\|\cdot\|_1=\|\cdot\|_{1,2}$ and $H^2(\mathcal{O})=\left\{u\in H^1(\mathcal{O});\ D_{ij}^2u\in L^2(\mathcal{O}),\ 1\leq i,j\leq d\right\},$ with its usual norm $\|\cdot\|_{H^2(\mathcal{O})}$. $H^{-1}(\mathcal{O})$ with the norm $\|\cdot\|_{-1}$ denotes the dual of $H_0^1(\mathcal{O})$. By $BV(\mathcal{O})$ we denote the space of functions $u$ of bounded variation on $\mathcal{O}$.

We set $A=-\Delta, \mathcal{D}(A)=H_0^1(\mathcal{O})\cap H^2(\mathcal{O}),$ that is, the Laplace operator associated to  Dirichlet boundary conditions. Then,  we consider an eigenbasis of  $L^2(\mathcal{O})$, denoted by $\left\{e_k\right\}_{k\in\mathbb{N}^*},\ e_k\in H^2(\mathcal{O})\cap H_0^1(\mathcal{O})$. Finally, for each $\epsilon>0$ we set   
\begin{equation}\label{Ae}J_\epsilon=(1+\epsilon A)^{-1},\ A_\epsilon=AJ_\epsilon=\frac{1}{\epsilon}(I-J_\epsilon),\end{equation}namely, the resolvent and the Yosida approximation of the Laplace operator, respectively.

 Next, we introduce $\mathcal{B}$, the set of all functions $b$ of the form $b=(b^1,...,b^d),\ b^i:\mathbb{R}^d \rightarrow \mathbb{R},\ i=1,...,d,$ such that 
\begin{itemize}
\item[\textbf{(H1)}] $b^i\in C^2(\overline{\mathcal{O}}),\ i=1,...,d$;
\item[\textbf{(H2)}] $\text{div}\ b=0$;
\item[\textbf{(H3)}] $b$ is tangent to the boundary $\partial\mathcal{O}$, of the domain $\mathcal{O}$.
\end{itemize}

Now, let any $b\in \mathcal{B}$. We associate to it the  operators  $B: H^1_0(\mathcal{O})\rightarrow L^2(\mathcal{O})$, defined as
\begin{equation}\label{momo1}Bv:=b\cdot\nabla v,\ \forall v\in H_0^1(\mathcal{O});\end{equation}and $B^2:H_0^1(\mathcal{O})\rightarrow L^2(\mathcal{O})$ 
$$B^2u=-B^*Bu,\ \forall u\in H_0^1(\mathcal{O}),$$where $B^*$ is the adjoint of $B$ in $L^2(\mathcal{O})$. In the following, we shall see that the domain of $B^*$ contains $H_0^1(\mathcal{O})$ and we have $B^*u=-Bu,\ \forall u\in H_0^1(\mathcal{O})$. Consequently, the above notation is meaningful.

  We know from \cite{brez}, p. 439 to p. 443, that, for $b\in \mathcal{B}$,   the linear operator $B$ has the following properties: 
\begin{itemize}
\item[(B1)]There exists a positive constant $c_1(b)$, such that
$$\|B\|_{L(H_0^1(\mathcal{O}),L^2(\mathcal{O}))}\leq c_1(b).$$
\item[(B2)] The adjoint of $B$ in $L^2(\mathcal{O})$, denoted by $B^*$, satisfies $\mathcal{D}(B^*)\supset H_0^1(\mathcal{O})$ and $B^*u=-Bu,\ \forall u\in H_0^1(\mathcal{O})$; so,  for all $u\in H_0^1(\mathcal{O})$ it follows that $\left<Bu,u\right>=0$ and $\left<u,B^2u\right>=-\left<Bu,Bu\right>=-|Bu|_2^2$ (which will be frequently used in the sequel).
\item[(B3)] The operator $B$ is the infinitesimal generator of a contraction $C_0-$group in $L^2(\mathcal{O})$, which we denote by $e^{sB},\ s\in\mathbb{R}.$

We include here a sketch of the proof for this point, since we will refer to it latter. The operator $B$ is m-dissipative, indeed, by the skew-adjointness, $B$ is dissipative; and for all $f\in L^2(\mathcal{O})$ the equation $u- Bu=f$ has the solution
\begin{equation}\label{noi8}u(\xi)=\int_0^\infty e^{-s}f(\zeta(s,\xi))d\xi,\ \forall \xi\in\mathcal{O},\end{equation}where $s\rightarrow \zeta(s,\xi)$ is the differential flow defined by the equation 
\begin{equation}\label{zetaa}\frac{d}{ds}\zeta=b(\zeta),\ s\geq 0;\ \zeta(0)=\xi.\end{equation}(By assumptions \textbf{ (H1)} and \textbf{(H3)}, it follows that $s\rightarrow \zeta(s,\xi)$ is well-defined on $[0,\infty)$, is of class $C^2-$in $\xi$ and preserves $\mathcal{O}$.) Hence, $B$ generates a $C_0-$group, $(e^{sB})_{s\in\mathbb{R}}$, on $L^2(\mathcal{O})$, which is given by
$$(e^{sB}f)(\xi)=f(\zeta(s,\xi)),\ \forall f\in L^2(\mathcal{O}),s\in\mathbb{R}.$$
\item[(B4)] Let any $s\in\mathbb{R}$, then, we have: $$e^{sB}y\geq0 \text{ for } y\geq0,$$
$$(e^{sB}y)( e^{sB}y^+)=(e^{sB}y^+)^2$$ and 
$$\nabla(e^{\beta(t)B}y)\cdot \nabla(e^{\beta(t)B}y^+)=|\nabla(e^{\beta(t)B}y^+|^2,$$ for all $y\in H_0^1(\mathcal{O})$, where $y^+$ stands for the positive part of $y$, and $\beta$ is, this time, some one-dimensional Brownian motion.

This is indeed so. By the definition of the group in (B3), we have for all $\xi\in\mathcal{O}$
$$e^{sB}y=y(\zeta(s,\xi))\geq0 \text{ if } y\geq0,$$
$$(e^{sB}y)(e^{sB}y^+)=y(\zeta(s,\xi))y^+(\zeta(s,\xi))=(y^+(\zeta(s,\xi))^2=[e^{sB}y^+]^2,$$and
$$\begin{aligned}\nabla(e^{\beta(t)B}y)\cdot \nabla(e^{\beta(t)B}y^+)&=\nabla(y(\zeta(\beta(t),\xi))\cdot\nabla y^+(\zeta(\beta(t),\xi))\\&
=|\nabla y^+(\zeta(\beta(t),\xi))|^2=|\nabla(e^{\beta(t)B}y^+|^2.\end{aligned}$$
\item[(B5)] $e^{sB}(H_0^1(\mathcal{O}))\subset H_0^1(\mathcal{O})$, and the restriction of $e^{sB}$ to $H_0^1(\mathcal{O})$ is a $C_0-$group in $H_0^1(\mathcal{O})$; besides this, $e^{sB}(H^2(\mathcal{O}))\subset H^2(\mathcal{O})$ (this is true from the definition of the group and because $\zeta$ is of class $C^2(\mathcal{O})$).
\item[(B6)] There exist constants $M(b)>0,\alpha(b)$, independent of $s$, such that 
$$\|e^{sB}\|_{\mathcal{L}(H_0^1(\mathcal{O}))}\leq M(b)e^{\alpha(b)|s|},\ \forall s\in\mathbb{R}.$$

\item[(B7)] We have $$ \left<e^{sB}u,v\right>=\left<u,e^{-sB}v\right>, \ \forall u,v \in H_0^1(\mathcal{O}),$$
$$|e^{sB}u|_2=|u|_2,\forall u\in L^2(\mathcal{O}),$$ and 
$$ \frac{1}{M(b)e^{\alpha(b)|s|}}|\nabla u|_2\leq |\nabla(e^{sB}u)|_2,\ \forall u\in H_0^1(\mathcal{O}),\ \forall s\in \mathbb{R}.$$

The first one follows by the fact that the adjoint of $e^{sB}$ is $e^{-sB}$, since $B^*=-B$; the second one follows by the fact that the Jacobian of $\zeta$ is equal to one and the definition of the group $e^{sB}$, while the last one can be deduced by equivalently writing
$$|\nabla u|_2=|\nabla[e^{-sB}(e^{sB}u)]|_2\leq M(b)e^{\alpha(b)|s|}|\nabla(e^{sB}u)|_2,$$ and using (B6).
\end{itemize}

Now, let $b_1,...,b_N$ be $N$ functions from $\mathcal{B}$. We assume two more hypotheses on them
$$\mathbf{(H_\Delta)} \ \ \ \ \ \ \ \ \ \ \ \ \ \  \ \ \ \ \ \   \Delta e^{sB_i}u=e^{sB_i}\Delta u,\ \forall u\in H^2(\mathcal{O})\cap H_0^1(\mathcal{O}),\ \forall i=1,2,...,N,\ \forall \epsilon >0.\ \ \ \ \ \ \ \  \ \ \ \ \ \ \ $$Since $e^{sB_i}$ preserves $H^2(\mathcal{O})\cap H_0^1(\mathcal{O})$, $\mathbf{(H_\Delta)}$ implies  $J_\epsilon e^{sB_i}u=e^{sB_i} J_\epsilon u$ for all $ u\in H_0^1(\mathcal{O}),\ \forall i=1,2,...,N,\ \forall \epsilon >0.$ Here, $B_i,\ i=1,...,N,$ are the associated operators of $b_i,\ i=1,...,N$ as in (\ref{momo1}); while $e^{sB_i},\ i=1,...,N$ are the $C_0-$groups generated by them, defined as in (B3). 

And

$$\mathbf{(H_C)}  \ \ \ \ \ \ \ \ \ \ \ \  \ \ \ \ \\ \ \ \\ \ \  \ \ \ \ \ \ \ \ \ \ \ \ \ \   \ e^{sB_i}e^{sB_j}=e^{sB_j}e^{sB_i},\ \forall i,j=1,...,N,\ \ \ \ \ \ \ \ \ \ \ \ \ \ \ \ \ \ \  \ \ \ \ \  \ \  \ \  $$ that is, the groups $e^{sB_i},\ i=1,...,N,$ commute, and so 
$$e^{sB_1}e^{sB_2}...e^{sB_N}=e^{s\sum_{i=1}^NB_i}.$$

Before moving on, let us give some examples of such $b_i$ that obey all the above hypotheses.
\begin{example}\label{ex} Let $\Lambda_1,...,\Lambda_N$  be $N$ skew-symmetric, mutually commuting matrices from $\mathcal{M}_d(\mathbb{R})$, i.e. 
$$\Lambda_i^\mathcal{T}=-\Lambda_i \text{ and }\Lambda_i\Lambda_j=\Lambda_j\Lambda_i,\  \text{for } i,j=1,...,N.$$  Assume that  $\left<\Lambda_i\xi,\nu(\xi)\right>_{\mathbb{R}^d}=0,\ \forall \xi\in \partial \mathcal{O}$, for all $i=1,2,...,N$; where $\nu$ is the unit outward normal of the boundary $\partial\mathcal{O}$.

We claim that $b_i(\xi):=\Lambda_i\xi,\ \xi\in\mathbb{R}^d,\ i=1,...,N,$ satisfy our assumptions. Indeed, \textbf{(H1)} and \textbf{(H3)} are obvious, while \textbf{ (H2)} follows by noticing that being  skew-symmetric, the matrices $\Lambda_i$ have the trace $Tr\Lambda_i=0,\ i=1,...,N.$ Furthermore,  the solution $\zeta_i$ to the equation
$$\frac{d}{ds}\zeta_i(s)=b_i(\zeta_i(s)),\ \zeta(0)=\xi;$$is given by $\zeta_i(s)\xi=e^{s\Lambda_i}\xi,\ i=1,...,N.$ It is easy to see that, for all $s\in\mathbb{R}$, $\zeta_i(s)$ are linear maps from $\overline{\mathcal{O}}$ to $\overline{\mathcal{O}}$, invertible, $(\zeta_i(s))^{-1}=e^{-s\Lambda_i}$, and $(\zeta_i(s))^\mathcal{T}=(e^{s\Lambda_i})^\mathcal{T}=e^{s\Lambda_i^\mathcal{T}}=e^{-s\Lambda_i}=(\zeta_i(s))^{-1}$. Therefore, $\zeta_i(s)$ are orthogonal linear transformations of $\overline{\mathcal{O}},$ and so we have the invariance of the Laplacean, namely for $f\in H^2(\mathcal{O})\cap H_0^1(\mathcal{O})$,
$$\Delta (f(\zeta_i(s)\xi))=\Delta f(\zeta_i(s)\xi),\ \forall \xi\in\mathcal{O},$$ or, equivalently 
$$\Delta e^{sB_i}f= e^{sB_i}\Delta f,$$that is exactly $\mathbf{(H_\Delta)}$. Finally, the mutual commutativity of $\Lambda_i,\ i=1,...,N,$ immediately implies $\mathbf{(H_C)}$. 

Before ending with this example, let us mention that $\Lambda_1 =\left(\begin{array}{cc}0 &1\\ -1& 0\end{array}\right),\ \Lambda_2 =\left(\begin{array}{cc}0 &-1\\ 1& 0\end{array}\right)$ and $\mathcal{O}:=\left\{(\xi_1,\xi_2)\in\mathbb{R}^2:\ \xi_1^2+\xi_2^2< R^2\right\}$ for some $R>0$, satisfy the above conditions; thus, $b_1(\xi)=\left(\begin{array}{c}\xi_2\\ -\xi_1\end{array}\right) $ and $b_2(\xi)=\left(\begin{array}{c}-\xi_2\\ \xi_1\end{array}\right) $ obey \textbf{(H1)}-\textbf{(H3)} together with $\mathbf{(H_\Delta)}$ and $\mathbf{(H_C)}$, on this particular $\mathcal{O}$. 
\end{example}

Next, for latter purpose, let us consider the function $v: \mathbb{R}^N\rightarrow \mathbb{R},$ $v(x):=e^{\sum_{i=1}^N x_i B_i}\varphi$, where $\varphi\in H_0^1(\mathcal{O})\cap L^\infty(\mathcal{O})$. By It\^o's formula applied to $v(\beta(t))$, we get
\begin{equation}\label{e2}\begin{aligned}e^{\sum_{i=1}^N\beta_i(t)B_i}\varphi=&\varphi+\frac{1}{2}\int_0^t e^{\sum_{i=1}^N\beta_i(s)B_i }\sum_{i=1}^NB_i^2\varphi ds
+\int_0^te^{\sum_{i=1}^N\beta_i(s)B_i}\sum_{i=1}^NB_i\varphi d\beta_i(s).\end{aligned}\end{equation}

Finally, we introduce the map $\phi: D(\phi)=BV(\mathcal{O})\cap L^2(\mathcal{O})\rightarrow \mathbb{R},$ as follows
\begin{equation}\label{e25}\phi(u)= \|Du\|+\int_{\partial\mathcal{O}}|\gamma_0(u)|d\mathcal{H}^{d-1},\ \forall u\in BV(\mathcal{O})\cap L^2(\mathcal{O})\end{equation} and put
$\phi(u)=+\infty$  if  $u\in L^2(\mathcal{O})\setminus BV(\mathcal{O})$. Here,
$$\|Du\|=\sup\left\{\int_\mathcal{O} u\ \text{div} \varphi d\xi:\ \varphi\in C_0^\infty(\mathcal{O};\mathbb{R}^d),\ |\varphi|_\infty\leq 1\right\},$$
$\gamma_0(u)$ is the trace of $u$ on the boundary and $d\mathcal{H}^{d-1}$ is the Hausdorff measure. Then, we define its subdifferential
$$\partial\phi(u)=\left\{\eta\in L^2(\mathcal{O}):\ \phi(u)-\phi(v)\leq \left<\eta,u-v\right>,\ \forall v\in D(\phi)\right\}.$$

Arguing likewise in \cite{b3}, we may rewrite equation (\ref{ioi20}) in the following equivalent form
\begin{equation}\label{e26}dX(t)+\partial\phi( X(t))dt \ni \frac{1}{2}\sum_{i=1}^NB_i^2(t) X(t)dt+\sum_{i=1}^NB_i(t) X(t)d\beta_i(t),\ t\geq0; \ X(0)=x.\end{equation}

Based on the above reformulation of the equation, we may give the definition of a stochastic variational solution for (\ref{e26}), equivalently for (\ref{ioi20}).
\begin{definition}\label{d1}Let $x\in L^2(\mathcal{O})$. A stochastic process $X:[0,T]\times \Omega\rightarrow L^2(\mathcal{O})$ is said to be a variational solution to (\ref{ioi20}) if the following conditions hold:
\begin{itemize}
\item[(i)] $X$ is $(\mathcal{F}_t)-$ adapted, has $\mathbb{P}-$a.s. continuous sample paths in $L^2(\mathcal{O})$ and $X(0)=x$;
\item[(ii)] $X\in L^2([0,T];\ L^2(\mathcal{O})),\ \phi(X)\in L^1([0,T];\ L^2(\mathcal{O}))\ \ \ \mathbb{P}-\text{a.s.}$;
\item[(iii)] for each $(\mathcal{F}_t)-$progressively measurable process $G\in L^2([0,T];\ L^2(\mathcal{O}))$ $\mathbb{P}-$a.s.  and each $(\mathcal{F}_t)-$ adapted $L^2(\mathcal{O})$-valued process $Z$ with $\mathbb{P}-$a.s. continuous sample paths such that $Z\in L^2([0,T];\ H_0^1(\mathcal{O}))$ $\mathbb{P}-$a.s. and solving the equation
\begin{equation}\label{e27}Z(t)-Z(0)+\int_0^tG(s)ds=\frac{1}{2}\int_0^t\sum_{i=1}^NB_i^2 Z(s)ds+\int_0^t \sum_{i=1}^NB_iZ(s)d\beta_i(s),\ t\in [0,T],\end{equation}we have
\begin{equation}\label{e28}\begin{aligned}\frac{1}{2}&|X(t)-Z(t)|^2_2+\int_0^t\phi(X(s))ds\leq \frac{1}{2}|x-Z(0)|^2_2\\&
+\int_0^t\phi(Z(s))ds+\int_0^t\left<G(s),X(s)-Z(s)\right>ds \ \mathbb{P}-\text{a.s.},\ t\in[0,T].\end{aligned}
\end{equation}Here, $\phi$ is defined in (\ref{e25}), $\left<\cdot,\cdot\right>$ is the duality pairing with pivot space $L^2(\mathcal{O}).$ (Notice that equation (\ref{e27}) has a unique solution for a given initial solution in $L^2(\mathcal{O})$, see \cite{brez}.)
\end{itemize}
\end{definition}
The relation between (\ref{ioi20}) and (\ref{e28})  becomes clearer if one applies (formally) the It\^o's formula to $\frac{1}{2}|X-Z|^2_2$ and take into account  the skew-adjointness of $B_i,\ i=1,...,N$ (see (\ref{lele1}) for details).

 The main existence result  is stated in the  theorem below.
\begin{theorem}\label{t1} Let $\mathcal{O}$ be a bounded and convex open subset of $\mathbb{R}^d$ with smooth boundary, $b_i\in \mathcal{B},\ i=1,...,N,$ such that hypotheses $\mathbf{(H_\Delta)}$ and $\mathbf{(H_C)}$ hold true; and $T>0$. For each $x\in L^2(\mathcal{O})$ there is a unique variational solution $X$ to equation (\ref{ioi20}), such that, for all $p\geq2$
\begin{equation}\label{e230}\sup_{t\in[0,T]}|X(t)|^p_2\leq |x|^p_2\ \ \  \mathbb{P}-\text{a.s.}. \end{equation} Furthermore, if $x,x^*\in L^2(\mathcal{O})$ and $X,X^*$ are the corresponding variational solutions with initial conditions $x,x^*$, respectively, then
\begin{equation}\label{e232}\sup_{t\in[0,T]}|X(t)-X^*(t)|^2_2\leq |x-x^*|^2_2 \ \ \  \mathbb{P}-\text{a.s.}.\end{equation} 
\end{theorem}

\section{The equivalent random partial differential equation}
The trick to prove Theorem \ref{t1}  is to rewrite equivalently equation (\ref{ioi20}) as a random differential equation, namely the following one
\begin{equation}\label{e240}\left\{\begin{array}{l}\begin{aligned}\partial_t Y(t)&=e^{-\sum_{i=1}^N\beta_i(t) B_i}\text{div}(\text{sgn}(\nabla(e^{\sum_{i=1}^N\beta_i(t) B_i}Y(t))))\ \mathbb{P}-\text{a.s. in } (0,T)\times \mathcal{O},\end{aligned}\\
Y=0 \text{ on }(0,T)\times \partial\mathcal{O},\\
Y(0,\xi)=x(\xi),\ \xi\in\mathcal{O},\end{array}\right.\ \end{equation}by the substitution $Y(t)=e^{-\sum_{i=1}^N\beta_i(t)B_i}X(t)$. This idea is due to \cite{brez}, which was also used in \cite{b5}. There  equations of similar form as (\ref{ioi20}) are treated, the main difference is that, in our case the corresponding leading operator is of high singularity, which is not the case in \cite{brez,b5}. Therefore, the equivalence and  all the other existence and uniqueness  results must be reconsidered and  proved in the new framework. More exactly, we shall apply the technique in \cite{b3}.

In order to rigorously show the equivalence between (\ref{ioi20}) and (\ref{e240}), the definition of the solution of the equation (\ref{e240}) must be given in the sense of a variational inequality, this time a deterministic one, however with random terms. More exactly,
\begin{definition}Let $x\in L^2(\mathcal{O})$. A stochastic process $Y:[0,T]\times\Omega\rightarrow L^2(\mathcal{O})$ is said to be a variational solution to (\ref{e240}) if the following conditions hold:
\begin{itemize}
\item[(i)] $Y$ is $(\mathcal{F}_t)-$ adapted, has $\mathbb{P}-$a.s. continuous sample paths in $L^2(\mathcal{O})$ and $Y(0)=x$;
\item[(ii)] $e^{\sum_{i=1}^N\beta_i B_i}Y\in L^2([0,T];\ L^2(\mathcal{O})),\ \phi(e^{\sum_{i=1}^N\beta_i B_i}Y)\in L^1([0,T];\ L^2(\mathcal{O}))$ $\mathbb{P}-$a.s.;
\item[(iii)] for each $(\mathcal{F}_t)-$progressively measurable process $G\in L^2([0,T];\ L^2(\mathcal{O}))$ $\mathbb{P}-$a.s.,  and $Z(0)\in \ L^2(\mathcal{O})$ $\mathbb{P}-$a.s., denote by
\begin{equation}\label{e250}Z(t)=Z(0)-\int_0^te^{-\sum_{i=1}^N\beta_i(s)B_i}G(s)ds.\end{equation}(So, $Z(t)$ is an $(\mathcal{F}_t)-$ adapted $L^2(\mathcal{O})$-valued process,  with $\mathbb{P}-$a.s. continuous sample paths such that $e^{\sum_{i=1}^N\beta_i B_i}Z\in L^2([0,T];\ H_0^1(\mathcal{O}))$ $\mathbb{P}-$a.s..) We have
\begin{equation}\label{e251}\begin{aligned}\frac{1}{2}&|e^{\sum_{i=1}^N\beta_i(t)B_i}(Y(t)-Z(t))|^2_2+\int_0^t\phi(e^{\sum_{i=1}^N\beta_i(s)B_i}Y(s))ds\leq \frac{1}{2}|x-Z(0)|^2_2\\&
+\int_0^t\phi(e^{\sum_{i=1}^N\beta_i(s)B_i}Z(s))ds+\int_0^t\left<G(s),e^{\sum_{i=1}^N\beta_i(s)B_i}(Y(s)-Z(s))\right>ds\ \mathbb{P}-\text{a.s.},\ t\in[0,T].\end{aligned}
\end{equation}Here, $\phi$ is defined in (\ref{e25}), $\left<\cdot,\cdot\right>$ is the duality pairing with pivot space $L^2(\mathcal{O}).$ 
\end{itemize}
\begin{remark}As before, the relation between (\ref{e251}) and (\ref{e240}) is evident once one applies (formally) the It\^o's formula to $\frac{1}{2}|Y(t)-Z(t)|^2_2$, and takes into account that, by (B7), $$|Y(t)-Z(t)|_2^2=|e^{\sum_{i=1}^N\beta_i(t)B_i}(Y(t)-Z(t))|_2^2.$$\end{remark}
\end{definition}
Now we claim the equivalence between the two equations
\begin{proposition}$X:[0,T]\times\Omega\rightarrow L^2(\mathcal{O})$ is a variational solution to equation (\ref{ioi20}) if and only if $Y:=e^{-\sum_{i=1}^N\beta_i B_i}X$ is a variational solution to (\ref{e240}).
\end{proposition}
The above proposition follows from Proposition \ref{1p1} (ii) below. In the sequel, it will be important to distinguish between the space $\mathcal{L}^2(\mathcal{O})$ of square integrable functions on $\mathcal{O}$, and $L^2(\mathcal{O})$ the corresponding $d\xi-$classes.

\begin{proposition}\label{1p1}Let $G\in L^2([0,T];\ L^2(\mathcal{O}))$ $\mathbb{P}-$a.s. be $(\mathcal{F}_t)-$progressively measurable and $Z(0)\in L^2(\Omega,\mathcal{F}_0;\ L^2(\mathcal{O})).$ Let $G^0$ be a $(dt \otimes d\xi\otimes \mathbb{P})-$ version of $G$ such that $(t,\omega)\rightarrow G^0(t,\xi,\Omega)$ is $(\mathcal{F}_t)-$progressively measurable and in $L^2([0,t]\times\Omega)$ for every $\xi\in\mathcal{O}$. Furthermore, let $Z^0$ be a $(d\xi\otimes\mathbb{P})-$version of $Z(0)$ such that $\omega\rightarrow Z^0(\xi,\omega)$ is $\mathcal{F}_0-$measurable for all $\xi\in\mathcal{O}$.
\begin{itemize}
\item[(i)] Define
\begin{equation}\label{313}\begin{aligned}Z(t):=& e^{\sum_{i=1}^N\beta_i(t)B_i}Z^0\\&
-e^{\sum_{i=1}^N\beta_i(t)B_i}\int_0^t e^{-\sum_{i=1}^N\beta_i(s)B_i}G^0(s) ds,\  t \in [0,T].\end{aligned}\end{equation}Then, $Z$ is solution to the stochastic differential equation
\begin{equation}\label{323}dZ(t)=-G^0(t)dt+\frac{1}{2}\sum_{i=1}^NB_i^2(t)Z(t)dt+\sum_{i=1}^NB_i(t)Z(t)d\beta_i(t),\ t\in[0,T];\ Z(0)=Z^0,\end{equation}which is $\mathcal{B}([0,t])\otimes\mathcal{B}(\mathcal{O})\otimes\mathcal{F}_t-$measurable for each $t\in[0,T]$. (Here $\mathcal{B}(O)$ is the Borel set associated to the  set $O$).

Furthermore, the map $t\rightarrow Z(t) \in L^2(\mathcal{O})$ is $\mathbb{P}-$a.s. continuous. Hence,   $Z(t)$ is the unique solution to (\ref{e27}).
\item[(ii)] An $(\mathcal{F}_t)-$ adapted $\mathbb{P}$-a.s. continuous $L^2(\mathcal{O})$-valued process $(Z(t))_{t\in[0,T]}$ is a solution to the stochastic equation (\ref{e27}) if and only if $(e^{-\sum_{i=1}^N\beta_i(t)B_i}Z(t)))_{t\in[0,T]}$ is a solution to the deterministic equation (\ref{e250}) for $\mathbb{P}-$a.e. given $\omega\in\Omega.$
\end{itemize}
\end{proposition} 
\begin{proof}Item (ii) is a direct consequence of (i); that is why we only prove (i).

 Via (\ref{313}) we get that
\begin{equation}\label{raw20}\frac{d}{dt}(e^{-\sum_{i=1}^N\beta_i(t)B_i}Z(t))=-e^{-\sum_{i=1}^N\beta_i(t)B_i}G^0(t),\ t\in [0,T].\end{equation}Next, consider a symmetric  mollifier $\rho_\epsilon,\ \epsilon>0,$ (that is, $\rho_\epsilon(\xi-\eta)=\rho_\epsilon(\eta-\xi)$) and, given a function $u$,  denote by $u_\epsilon$ its convolution with it. Notice that we have 
$$\left<u_\epsilon,v\right>=\left<u,v_\epsilon\right>,\ \forall u,v\in L^2(\mathcal{O}).$$ Further, let any $\varphi\in H_0^1(\mathcal{O})$, we then have
$$\left<(e^{-\sum_{i=1}^N\beta_i(t)B_i}Z(t))_\epsilon,\varphi\right>=\left<Z(t),e^{\sum_{i=1}^N\beta_i(t)B_i}\varphi_\epsilon\right>,$$and so
$$\begin{aligned}\left<\left(\frac{d}{dt}(e^{-\sum_{i=1}^N\beta_i(t)B_i}Z(t))\right)_\epsilon,\varphi\right>=&\left<d Z(t),e^{\sum_{i=1}^N\beta_i(t)B_i}\varphi_\epsilon\right>+\left<Z(t),d(e^{\sum_{i=1}^N\beta_i(t)B_i}\varphi_\epsilon)\right>\\&+\int_\mathcal{O}dZ(t)\cdot d(e^{\sum_{i=1}^N\beta_i(t)B_i}\varphi_\epsilon)d\xi,\end{aligned}$$where the above product $\cdot$ is the formal It\^o's product between two stochastic differentials. Taking into account that $Z$ is a semi-martingale, we may denote by $dZ(t)=:\mu(t)dt+\sum_{i=1}^N\sigma_i(t)d\beta_i(t)$, then, recalling relation (\ref{e2}), the above equality implies that
$$\begin{aligned}&\left<\left(\frac{d}{dt}(e^{-\sum_{i=1}^N\beta_i(t)B_i}Z(t))\right)_\epsilon,\varphi\right>=\left<(e^{-\sum_{i=1}^N\beta_i(t)B_i}d Z(t))_\epsilon,\varphi\right>\\&
+\left<Z(t),\frac{1}{2}e^{\sum_{i=1}^N\beta_i(t)B_i}\sum_{i=1}^NB_i^2\varphi_\epsilon\right>+\left<Z(t),e^{\sum_{i=1}^N\beta_i(t)B_i}\sum_{i=1}^NB_i\varphi_\epsilon d\beta_i(t)\right>\\&
+\sum_{i=1}^N\left<\sigma_i(t),e^{\sum_{i=1}^N\beta_i(t)B_i}B_i\varphi_\epsilon\right>\end{aligned},$$that yields
$$\begin{aligned}&\left<\left(\frac{d}{dt}(e^{-\sum_{i=1}^N\beta_i(t)B_i}Z(t)\right)_\epsilon,\varphi\right>=\left<(e^{-\sum_{i=1}^N\beta_i(t)B_i}d Z(t))_\epsilon,\varphi\right>\\&
\left<\frac{1}{2}\left[\sum_{i=1}^NB_i^2e^{-\sum_{i=1}^N\beta_i(t)B_i}Z(t)\right]_\epsilon,\varphi\right>\\&
-\sum_{j=1}^N\left<\left(B_je^{-\sum_{i=1}^N\beta_i(t)B_i}Z(t)\right)_\epsilon d\beta_j(t),\varphi\right>\\&
-\sum_{j=1}^N\left<\left(B_je^{-\sum_{i=1}^N\beta_i(t)B_i}\sigma_j(t)\right)_\epsilon,\varphi\right>,\ \forall \varphi\in H_0^1(\mathcal{O}).\end{aligned}$$Here we have frequently  used the fact that $B_i,\ i=1,...,N,$  are skew-adjoint.

Since the mollified functions are continuous in $\xi$,  taking $\epsilon_n=\frac{1}{n}$ and letting $n \rightarrow \infty$, we arrive to
$$\begin{aligned}&\frac{d}{dt}(e^{-\sum_{i=1}^N\beta_i(t)B_i}Z(t))=e^{-\sum_{i=1}^N\beta_i(t)B_i}d Z(t)+\frac{1}{2}\sum_{i=1}^NB_i^2e^{-\sum_{i=1}^N\beta_i(t)B_i}Z(t)\\&
-\sum_{j=1}^NB_je^{-\sum_{i=1}^N\beta_i(t)B_i}Z(t) d\beta_j(t)-\sum_{j=1}^NB_je^{-\sum_{i=1}^N\beta_i(t)B_i}\sigma_j(t),\ t\in [0,T], \mathbb{P}\otimes d\xi -\text{a.s.},\end{aligned}$$where using relation (\ref{raw20}), we arrive to the fact that $Z$ satisfies the following stochastic differential equation 
$$dZ(t)=-G^0(t)dt+\frac{1}{2}\sum_{i=1}^NB_i^2(t)Z(t)dt+\sum_{i=1}^NB_i(t)Z(t)d\beta_i(t),\ t\in[0,T],\ \mathbb{P}\otimes d\xi-\text{a.s.},$$which means that $Z$ indeed satisfies (\ref{323}).

Finally, let any $b\in\mathcal{B}$,  $B$ the associated operator as in (\ref{momo1}), and $\beta$ an one-dimensional Brownian motion. Moreover,  let $Z^1(t),\ Z^2(t)\in \mathcal{L}^2(\mathcal{O})$ such that $Z^1=Z^2$, in the sense that they belong to the same $d\xi$-class. Let any $i\in \mathbb{N}^*$ and $e_i$ the $i-$th vector from the eigenbasis of the Laplacean considered in the Preliminaries. Then, $\mathbb{P}-$a.s., for every $t\in[0,T]$, we have
$$\begin{aligned}\left<e_i,\int_0^t BZ^1(s)d\beta(s)\right>&=\int_0^t\left<e_i, BZ^1(s)\right>d\beta(s)\\&
=-\int_0^t\left<Be_i,Z^1(s)\right>d\beta(s)=-\int_0^t\left<Be_i,Z^2(s)\right>d\beta(s)\\&
=\int_0^t\left<e_i,BZ^2(s)\right>d\beta(s)=\left<e_i,\int_0^t BZ^2(s)d\beta(s)\right>\end{aligned},$$relying on the stochastic Fubini theorem and the skew-adjointness of $B$. The above means that, $\mathbb{P}-$a.s., $ \int_0^tBZ^1(s)d\beta(s) =\int_0^tBZ^2(s)d\beta(s).$ The same can be said for the integral $\int_0^tB^2Z^1(s)ds$ related to  $\int_0^t B^2Z^2(s)ds$. In conclusion, the above $Z$ is the unique solution to (\ref{323}).
\end{proof}

\subsection{Proof of the main existence and uniqueness result}

\textit{Proof of Theorem \ref{t1}}
 
\textit{\textbf{Existence}}.
As in \cite{b3} the approach is based on the  construction of approximating schemes for both equations (\ref{ioi20}) and (\ref{e240}). To this end, let $\lambda\in(0,1]$ be fixed, and introduce the Yosida approximation, $\psi_\lambda(u)$, of the function $\psi(u)=\text{sgn}(u),\ u\in\mathbb{R}^d$, that is
\begin{equation}\label{e260}\psi_\lambda(u)=\left\{\begin{array}{cc}\frac{1}{\lambda}u &, \text{ if }|u|\leq\lambda,\\
\frac{u}{|u|}&,\text{ if }|u|>\lambda.\end{array}\right.\ \end{equation}  For latter purpose, we also introduce the Moreau-Yosida approximation of the function $u\rightarrow |u|$, that is $j_\lambda(u)=\inf_v\left\{\frac{|u-v|^2}{2\lambda}+|v|\right\}$ and recall that we have $\nabla j_\lambda=\psi_\lambda,\ \forall \lambda>0$ (see, for instance, \cite{b2}). Finally, denote by $\tilde{\psi}_\lambda(u)=\psi_\lambda(u)+\lambda u,\ \forall u\in\mathbb{R}^d$.

Now, we approximate (\ref{ioi20}) by
\begin{equation}\label{e}\left\{\begin{array}{l}d X_\lambda(t)=\text{div}\tilde{\psi}_\lambda(\nabla X_\lambda(t))dt+\frac{1}{2}\sum_{i=1}^NB_i^2 X_\lambda(t))dt+\sum_{i=1}^NB_i X_\lambda(t)d\beta_i(t) \text{ in }(0,T)\times\mathcal{O},\\
X_\lambda=0 \text{ on }(0,T)\times \partial\mathcal{O},\\
X_\lambda(0)=x \text{ in }\mathcal{O},\end{array}\right.\ \end{equation}and the corresponding rescaled equation (\ref{e240}) by
\begin{equation}\label{ee9}\left\{\begin{array}{l}\frac{d}{dt} Y_\lambda(t)=e^{-\sum_{i=1}^N\beta_i(t)B_i}\text{div}\tilde{\psi}_\lambda(\nabla(e^{\sum_{i=1}^N\beta_i(t)B_i}  Y_\lambda(t))) \text{ in }(0,T)\times\mathcal{O},\\
Y_\lambda=0 \text{ on }(0,T)\times \partial\mathcal{O},\\
Y_\lambda(0)=x \text{ in }\mathcal{O}.\end{array}\right.\  \ \ \ \ \\ \ \\ \ \ \ \ \ \  \ \ \ \ \ \ \ \ \ \  \ \ \ \ \ \ \end{equation}

The proposition below is concerned on the existence of solutions for (\ref{e}) and (\ref{ee9}), respectively, as-well on the equivalence between them.
\begin{proposition}\label{la3}\begin{itemize}
\item[(i)] For each $\lambda\in(0,1]$ and each $x\in L^2(\mathcal{O})$, there is a unique function $X_\lambda$,  which satisfies: $X_\lambda(0)=x,$ is $\mathbb{P}-$a.s. continuous in $L^2(\mathcal{O})$ and $(\mathcal{F}_t)-$adapted such that

\begin{equation}\label{e261}\begin{aligned}&X_\lambda\in L^2([0,T];\ H_0^1(\mathcal{O}))\ \ \mathbb{P}-\text{a.s.},\\&
X_\lambda(t)= x+\int_0^t \text{div}\   \tilde{\psi}_\lambda(\nabla X_\lambda(s))ds+ \frac{1}{2}\int_0^t\sum_{i=1}^NB_i^2 X_\lambda(s)ds\\&
+\int_0^t\sum_{i=1}^N B_iX_\lambda(s)d\beta_i(s),\ t\in[0,T],\ \mathbb{P}-\text{a.s.}.\end{aligned}\end{equation}Furthermore, we have

\begin{equation}\label{lele1}\frac{1}{2}|X_\lambda(t)|_2^2=\frac{1}{2}|x|_2^2-\int_0^t\left<\tilde{\psi}_\lambda(\nabla  X_\lambda(s)),\nabla X_\lambda(s)\right>ds,\ \forall t\in [0,T]\ \ \ \mathbb{P}-\text{a.s.}.\end{equation}

In particular, we have
 
\begin{equation}\label{e262}|X_\lambda(t)|_2\leq |x|_2 \ \ \ \mathbb{P}-\text{a.s.},\ \forall t\in[0,T];\end{equation} and, if $x,x^*\in L^2(\mathcal{O})$ and $X_\lambda,\ X_\lambda^*$ are the corresponding solutions with initial conditions $x,x^*$, respectively, then
\begin{equation}\label{e263}|X_\lambda(t)-X_\lambda^*(t)|_2\leq |x-x^*|_2\ \ \ \mathbb{P}-\text{a.s.},\ \forall t\in[0,T].\end{equation}

\item[(ii)] If $x\in H_0^1(\mathcal{O})$, then $\mathbb{P}-$a.s. equation (\ref{ee9}) has a unique solution such that
\begin{equation}\label{e265}Y_\lambda\in C([0,T];\ H_0^1(\mathcal{O})) \cap   L^\infty([0,T];\ H_0^1(\mathcal{O})) \cap L^2([0,T];\ H^2(\mathcal{O})).\end{equation}

\item[(iii)] $X_\lambda=e^{\sum_{i=1}^N\beta_i(t)B_i}Y_\lambda$ is an $(\mathcal{F}_t)-$adapted process with $\mathbb{P}-$a.s. continuous paths which is the unique solution of (\ref{e}), and we have  $\mathbb{P}-$a.s. $X_\lambda\in  C([0,T];\ H_0^1(\mathcal{O}))\cap  L^2([0,T];\ H^2(\mathcal{O})) \cap  L^\infty([0,T];\ H_0^1(\mathcal{O}))$. More exactly, we have
\begin{equation}\label{mos}\left[\prod_{i=1}^N\frac{1}{M(b_i)e^{\alpha(b_i)|\beta_i(t)|}}\|X_\lambda(t)\|_1\right]+2\lambda\int_0^t|\Delta X_\lambda(s)|_2^2\leq \|x\|_1^2, \  \forall t\in[0,T]\ \  \mathbb{P}-\text{a.s.}.
\end{equation}

\end{itemize}
\end{proposition}

\begin{proof}
\begin{itemize}
\item[(i)]Let us consider the operator $\mathcal{A}_\lambda:H_0^1(\mathcal{O})\rightarrow H^{-1}(\mathcal{O})$ defined by
\begin{equation}\label{e267}\left<\mathcal{A}_\lambda y,\varphi\right>=\int_\mathcal{O}\tilde{\psi}_\lambda(\nabla y)\cdot \nabla\varphi\ d\xi,\ \forall \phi\in H_0^1(\mathcal{O}).\end{equation}Hence, equation (\ref{e}) can be rewritten as
\begin{equation}\label{e268}d X_\lambda(t)+\mathcal{A}_\lambda X_\lambda(t)dt=\frac{1}{2}\sum_{i=1}^NB_i^2X_\lambda(t)dt+\sum_{i=1}^NB_iX_\lambda(t)d\beta_i(t),\ t\in[0,T];\ X_\lambda(0)=x,\end{equation}It is shown, for example, in \cite{b2} that $\mathcal{A}_\lambda$ is  demi continuous and it satisfies
$$\|\mathcal{A}_\lambda y\|_{-1}\leq \lambda\|y\|_1+\left(\int_\mathcal{O}d\xi\right)^\frac{1}{2},\ \forall y\in H_0^1(\mathcal{O}),$$and
$$\left<\mathcal{A}_\lambda y_1-\mathcal{A}_\lambda y_2,y_1-y_2\right>\geq \lambda\|y_1-y_2\|_1^2, \forall y_1,y_2\in H_0^1(\mathcal{O}).$$ Then, using similar arguments as in \cite{br2},  one may  deduce that the  equation (\ref{e268}) (equivalently, (\ref{e})) has a unique  solution, $X_\lambda$, satisfying the It\^o integral equation in (\ref{e261}).

Now applying It\^o's formula in (\ref{e261}) to the $L^2-$norm $\frac{1}{2}|X_\lambda(t)|^2_2$, we get
$$\begin{aligned}\frac{1}{2}|X_\lambda(t)|_2^2=&\frac{1}{2}|x|_2^2-\int_0^t\left<\mathcal{A}_\lambda X_\lambda(s),X_\lambda(s)\right>ds+\int_0^t\sum_{i=1}^N\left<X_\lambda(s),B_iX_\lambda(s)\right>d\beta_i(s)\\&
+\frac{1}{2}\int_0^t\sum_{i=1}^N\left(\left<B_i^2X_\lambda(s),X_\lambda(s)\right>+|B_iX_\lambda(s)|_2^2\right)ds\\&
\text{(where using  the skew-adjointness of $B_i$, see (B2))}\\&
=\frac{1}{2}|x|_2^2-\int_0^t\left<\mathcal{A}_\lambda X_\lambda(s),X_\lambda(s)\right>ds\\&
\text{(where using the monotonicity of $\mathcal{A}_\lambda$)}\\&
\leq\frac{1}{2}|x|_2^2,\end{aligned}$$from where relations (\ref{lele1}) and  (\ref{e262}) follow immediately. Similarly, one may show (\ref{e263}) as-well.

\item[(ii)] 

Let us denote by $\Gamma=\Gamma(t,\omega): H_0^1(\mathcal{O})\rightarrow H^{-1}(\mathcal{O})$, the operator defined as
$$\left<\Gamma(t,\omega) y,\varphi\right>=\left<\tilde{\psi}_\lambda(\nabla e^{\sum_{i=1}^N\beta_i(t)B_i}y),\nabla(e^{\sum_{i=1}^N\beta_i(t)B_i}\varphi)\right>,\ \forall y,\varphi\in H_0^1(\mathcal{O}).$$Then, equation (\ref{ee9}) can be rewritten as 
$$\frac{d}{dt}Y_\lambda(t)+\Gamma(t)Y_\lambda(t)=0.$$It is easy to check that for all $t\in[0,T],\ \omega\in\Omega$,  $\Gamma(t,\omega)$ is demi-continuous, and
$$\left<\Gamma(t,\omega)y_1-\Gamma(t,\omega)y_2,y_1-y_2\right>\geq \lambda\|y_1-y_2\|_1^2,\ \forall y_1,y_2\in H_0^1(\mathcal{O}).$$ So, immediately one may deduce the existence and uniqueness of a solution for (\ref{ee9}).

The rest of this item follows by the next two lemmas.
\begin{lemma}\label{lu1}Let $x\in H_0^1(\mathcal{O}).$ The solution $Y_\lambda$ to (\ref{ee9}) belongs to $ L^\infty([0,T];H_0^1(\mathcal{O}))\cap L^2([0,T];H^2(\mathcal{O}))$ and
\begin{equation}\label{e3e}\mathrm{ess}\sup_{t\in[0,T]}\|Y_\lambda(t)\|^2_1+2\lambda\int_0^T|\Delta Y_\lambda(t)|^2_2dt\leq \|x\|_1^2,\ \lambda\in(0,1],\ \ \mathbb{P}-\text{a.s.}.\end{equation}
\end{lemma}
\begin{proof}

Recall the operator $A_\epsilon$ introduced in (\ref{Ae}) and denote by $A_\epsilon^\frac{1}{2}$ its square root operator. By hypothesis $\mathbf{(H_\Delta)}$, we have
\begin{equation}\label{sosull2}\begin{aligned}&\left<e^{-\sum_{i=1}^N\beta_i(t)B_i}\text{div}\ \psi_\lambda(\nabla (e^{\sum_{i=1}^N\beta_i(t)B_i}u,A_\epsilon u\right>=\left<\text{div}\  \psi_\lambda(\nabla(e^{\sum_{i=1}^N\beta_i(t)B_i}u), A_\epsilon e^{\sum_{i=1}^N\beta_i(t)B_i(t)}u\right>\\&
+\frac{1}{\epsilon}\left<\text{div}\  \psi_\lambda(\nabla(e^{\sum_{i=1}^N\beta_i(t)B_i}u),e^{\sum_{i=1}^N\beta_i(t)B_i }J_\epsilon u-J_\epsilon e^{\sum_{i=1}^N\beta_i(t)B_i}u \right>\\&
=\left<\text{div}\  \psi_\lambda(\nabla(e^{\sum_{i=1}^N\beta_i(t)B_i}u), A_\epsilon e^{\sum_{i=1}^N\beta_i(t)B_i(t)}u\right>\geq0 \end{aligned}\end{equation}by similar arguments as relation (5.19) from \cite{b3}. Besides this, we know that
\begin{equation}\label{sosull3}\left<Au,A_\epsilon u\right>\geq |A_\epsilon u|_2^2,\ \forall u\in H_0^1(\mathcal{O}).\end{equation}

Now, multiplying scalarly, in $L^2(\mathcal{O}),$ equation (\ref{ee9}) by $A_\epsilon Y_\lambda$,  use  relations (\ref{sosull2}) and (\ref{sosull3}), we obtain 
$$\frac{1}{2}|A_\epsilon^\frac{1}{2}Y_\lambda(t)|_2^2+\lambda \int_0^t|A_\epsilon Y_\lambda(s)|_2^2ds\leq \frac{1}{2}|A_\epsilon^\frac{1}{2}x|_2^2,\ t\in[0,T]. $$Letting $\epsilon\rightarrow 0$ we arrive to the conclusion of the lemma.
\begin{lemma} Let $x\in H_0^1(\mathcal{O})$. Then, the solution $Y_\lambda$ to (\ref{ee9}) belongs to $C([0,T];H_0^1(\mathcal{O})),$ $\mathbb{P}-$a.s..
\end{lemma}
\begin{proof}
Since $\Delta$ commutes with $e^{sB_i}$ for all $s\in\mathbb{R}$ and $i=1,...,N$, we may rewrite equation (\ref{ee9}) as
\begin{equation}\label{raw22}\frac{d}{dt}Y_\lambda(t)=\lambda\Delta Y_\lambda(t)+f(t) \text{ in } (0,T)\times\mathcal{O},\end{equation}where $f(t):=e^{-\sum_{i=1}^N\beta_i(t)B_i}\text{div}\psi_\lambda(\nabla(e^{\sum_{i=1}^N\beta_i(t)B_i} Y_\lambda(t)))$.
Next, taking into account that, for $y\in H_0^1(\mathcal{O})\cap H^2(\mathcal{O}),$
$$\text{div}\psi_\lambda(\nabla y)=\left\{\begin{array}{cc}\frac{1}{\lambda}\Delta y & \text{on}\ \left\{|\nabla y|\leq\lambda\right\}\\
\frac{\Delta y}{|\nabla y|}-\frac{\nabla y\cdot\nabla|\nabla y|}{|\nabla y|^2}&\text{on}\ \left\{|\nabla y|>\lambda\right\},\end{array}\right.\ $$ and that $e^{sB_i}$ preserves $H_0^1(\mathcal{O})\cap H^2(\mathcal{O})$ for all $s\in\mathbb{R},\ i=1,...,N$,  by Lemma \ref{lu1} we have that $f \in L^2(0,T;\ L^2(\mathcal{O})),\ \mathbb{P}-\text{a.s.}.$ Then, classical theory on the heat equation leads to the wanted conclusion.

\end{proof}
\item[(iii)]

Let $\varphi\in H_0^1(\mathcal{O})\cap L^\infty(\mathcal{O})$. We have
\begin{equation}\label{e1}\begin{aligned}e^{\sum_{i=1}^N\beta_i(t)B_i}Y_\lambda(t)=\sum_{j=1}^\infty \left<Y_\lambda(t),e_j\right>e^{\sum_{i=1}^N\beta_i(t)B_i}e_j
\end{aligned}\end{equation}(Here, $(e_j)_{j\in\mathbb{N}^*}$ is the eigenbases  of the Laplacian considered in the Preliminaries section.) By (\ref{ee9}) and (\ref{e2}) it yields
$$\begin{aligned}&\left<Y_\lambda(t),e_j\right>e^{\sum_{i=1}^N\beta_i(t)B_i}e_j=\left<x,e_j\right>e_j\\&
+\int_0^t\left<e^{-\sum_{i=1}^N\beta_i(s)B_i}\text{div}\tilde{\psi}_\lambda(\nabla(e^{\sum_{i=1}^N\beta_i(s)B_i}Y_\lambda(s))),e_j\right>e^{\sum_{i=1}^N\beta_i(s)B_i }e_jds\\&
+\frac{1}{2}\int_0^t\left<Y_\lambda(s),e_j\right>e^{\sum_{i=1}^N\beta_i(s)B_i }\left(\sum_{i=1}^NB_i^2e_j
\right) ds
+\int_0^t\left<Y_\lambda(s),e_j\right>e^{\sum_{i=1}^N\beta_i(s)B_i}\left(\sum_{i=1}^NB_ie_j\right)d\beta_i(s), \end{aligned}$$for all $j\in\mathbb{N}$, by using the stochastic Fubini Theorem. Next, we sum the above equation from $j=1$ to $\infty$,   to obtain
$$\begin{aligned}e^{\sum_{i=1}^N\beta_i(t)B_i}Y_\lambda(t)=&x+\int_0^t \text{div}\tilde{\psi}_\lambda(\nabla e^{\sum_{i=1}^N\beta_i(s)B_i}Y_\lambda(s)ds+\frac{1}{2}\sum_{j=1}^N\int_0^t B_j^2e^{\sum_{i=1}^N\beta_i(s)B_i}Y_\lambda(s)ds\\&
+\sum_{j=1}^NB_je^{\sum_{i=1}^N\beta_i(s)B_i}Y_\lambda(s)d\beta_j(s),\end{aligned}$$
which  leads to the fact that $X_\lambda=e^{\sum_{i=1}^N\beta_iB_i}Y_\lambda$ solves (\ref{e}).

We notice that we were able to interchange the sums with the integrals because $ e^{-\sum_{i=1}^N\beta_iB_i}\text{div} \tilde{\psi}_\lambda(\nabla e^{\sum_{i=1}^N\beta_iB_i}Y_\lambda)$  belongs to $L^2([0,T]\time\Omega;L^2(\mathcal{O}))$ by Lemma \ref{lu1}.

Now, by (\ref{e3e}), we have
$$ \|e^{-\sum_{i=1}^N\beta_i(t)B_i}X_\lambda(t)\|_1^2+2\lambda\int_0^t|\Delta(e^{-\sum_{i=1}^N\beta_i(s)B_i} X_\lambda(s)|_2^2\leq \|x\|_1,\ \forall t\in [0,T],$$ where
using (B7) and the commutativity between $\Delta$ and the group $e^{sB_i},\ s\in\mathbb{R},\ i=1,...,N$, relation (\ref{mos}) follows immediately.

\end{proof}

\end{itemize}
\end{proof} 
\textbf{Continuation of the proof of Theorem \ref{t1}.} By the density of $H_0^1(\mathcal{O})$ in $L^2(\mathcal{O})$, it is enough to prove the existence for initial conditions $x\in H_0^1(\mathcal{O})$.

We shall show that the sequence $(X_\lambda)_\lambda$ is Cauchy in $ C([0,T];L^2(\mathcal{O}))$ $\mathbb{P}-$a.s., from where it will follow that there is $X$ such that
\begin{equation}\label{e8e}\lim_{\lambda\rightarrow 0}\left[\sup_{t\in[0,T]}|X_\lambda(t)-X(t)|_2^2\right]=0\ \ \mathbb{P}-\text{a.s.}.\end{equation}

By It\^o's formula in (\ref{e}) (see relation (\ref{lele1})), we have
$$\begin{aligned}&|X_\lambda(t)|^2_2= |x|_2^2+2\int_0^t\left<\text{div}\tilde{\psi}_\lambda(\nabla X_\lambda(s)),X_\lambda(s)\right>ds
 \end{aligned}$$ where, using the fact that $\tilde{\psi}_\lambda(u)\cdot u\geq j_\lambda(u)+\lambda |u|^2,\ \forall u\in\mathbb{R}^d$, it yields that
\begin{equation}\label{e5e}|X_\lambda(t)|_2^2+2\int_0^t\int_\mathcal{O}j_\lambda(\nabla X_\lambda(s))d\xi ds+2\lambda\int_0^t|\nabla X_\lambda(s)|_2^2ds\leq |x|_2^2,\ \forall \lambda>0,\ t\in[0,T].\end{equation}

Let $\lambda,\epsilon\in(0,1]$, and $X_\lambda,\ X_\epsilon$ the corresponding  solutions to (\ref{e}). By It\^o's formula (similarly as in (\ref{lele1})),  it follows that,
$$\begin{aligned}\frac{1}{2}&d|X_\lambda(t)-X_\epsilon(t)|_2^2+\left<\psi_\lambda(\nabla X_\lambda(t))-\psi_\epsilon(\nabla X_\epsilon(t)),\nabla X_\lambda(t)-\nabla X_\epsilon(t)\right>dt\\&
+\left<\lambda\nabla X_\lambda(t)-\epsilon \nabla X_\epsilon(t),\nabla(X_\lambda(t)-X_\epsilon(t))\right>dt=0, \ t\in[0,T].\end{aligned}$$
Taking into account that, by the definition of $\psi_\lambda$, we have (for details, see \cite{b3}, p. 817, lines 11 to 16)
$$(\psi_\lambda(u)-\psi_\epsilon(v))\cdot (u-v)\geq -(\lambda+\epsilon),$$and
$$\begin{aligned}&\left<\lambda\nabla X_\lambda(t)-\epsilon \nabla X_\epsilon(t),\nabla(X_\lambda(t)-X_\epsilon(t))\right>\\&
\geq -\left(\lambda^2|\Delta X_\lambda(t)|_2^2+\epsilon^2|\Delta X_\epsilon(t)|_2^2\right)-\frac{1}{2}|X_\lambda(t)-X_\epsilon(t)|_2^2,\end{aligned}$$ we deduce that
$$\begin{aligned}&|X_\lambda(t)-X_\epsilon(t)|_2^2\leq C \int_0^t|X_\lambda(s)-X_\epsilon(s)|_2^2ds+2(\lambda+\epsilon)t\int_\mathcal{O}d\xi\\&
+2\lambda^2\int_0^t|\Delta X_\lambda(s)|_2^2ds+2\epsilon^2\int_0^t|\Delta X_\epsilon(s)|^2_2ds,\ t\in[0,T].\end{aligned}$$Hence, via Gronwall's lemma and (\ref{mos}), for some constant $C>0$, we have
$$\sup_{0\leq s\leq t}|X_\lambda(s)-X_\epsilon(s)|_2^2\leq C(\lambda+\epsilon) \ \ \mathbb{P}-\text{a.s.},$$that is, the sequence $\left\{X_\lambda\right\}_\lambda$ is Cauchy in $ C([0,T];L^2(\mathcal{O}))$ $\ \mathbb{P}-\text{a.s.}$, and so, relation (\ref{e8e}) holds.

Recalling that $\phi$ is lower-semicontinuous in $L^1(\mathcal{O})$ (see (\ref{e25})), we have by (\ref{e8e}) and Fatou's lemma that
\begin{equation}\label{raw2}\liminf_{\lambda\rightarrow 0}\int_0^t\phi(X_\lambda(s))ds\geq\int_0^t\phi(X(s))ds,\ \forall t\in[0,T].\end{equation}

We know that 
\begin{equation}\label{raw3}|j_\lambda(\nabla u)-|\nabla u||\leq \frac{1}{2}\lambda,\end{equation} that yields
\begin{equation}\label{noi21}\left|\int_0^t\int_\mathcal{O}j_\lambda(\nabla X_\lambda(s))d\xi ds-\int_0^t\phi(X_\lambda(s))ds\right|\leq c\lambda.\end{equation} Hence, via (\ref{raw2}), we get
\begin{equation}\label{raw1}\int_0^t\phi(X(s))ds\leq \liminf_{\lambda\rightarrow0}\int_0^t\int_\mathcal{O}j_\lambda(\nabla(X_\lambda(s)))d\xi ds<\infty.\end{equation}

We point out that by (\ref{e263}) and (\ref{e8e}), relation (\ref{e232}) follows immediately; while, Fatou's lemma together with relations (\ref{e8e}) and (\ref{e262}) imply (\ref{e230}). 

It remains to prove (\ref{e28}). To this end, for all processes $Z$ as in Definition \ref{d1} (iii), by It\^o's formula, we get
\begin{equation}\label{noi20}\begin{aligned}\frac{1}{2}&|X_\lambda(t)-Z(t)|^2_2+\int_0^t\int_\mathcal{O}j_\lambda(\nabla X_\lambda(s))d\xi ds\leq \frac{1}{2}|x-Z(0)|^2_2\\&
+\int_0^t\int_\mathcal{O}j_\lambda (\nabla Z(s))d\xi ds+\int_0^t\left<G(s),X_\lambda(s)-Z(s)\right>ds,\ t\in[0,T].\end{aligned}
\end{equation}
We let $\lambda$ tend to zero and use relations (\ref{raw3}), (\ref{raw1}) and (\ref{e8e}) to see that (\ref{e28}) holds true.

\textbf{\textit{Uniqueness}}. Let $x^*\in L^2(\mathcal{O})$ and $x\in H_0^1(\mathcal{O})$. Let $X^*$ be a variational solution to (\ref{ioi20}), with $X^*(0)=x^*$; and $X$ be the solution constructed in the existence part, with $X(0)=x.$ Set $Y^*:=e^{-\sum_{i=1}^N\beta_i B_i}X^*$ and $Y:=e^{-\sum_{i=1}^N\beta_i B_i}X$. Moreover, set $Y_\lambda^\epsilon:=J_\epsilon(Y_\lambda)$, where $Y_\lambda$ is the solution to (\ref{ee9}). By Lemma \ref{lu1}, (B6) and \cite[Remark 8.2]{b3}, it follows that 
$$\begin{aligned}|\nabla e^{\sum_{i=1}^N\beta_i(t) B_i}Y_\lambda^\epsilon(t)|_2^2=&|\nabla J_\epsilon(e^{\sum_{i=1}^N\beta_i(t) B_i}Y_\lambda(t))|_2^2
\leq |\nabla e^{\sum_{i=1}^N\beta_i(t) B_i}Y_\lambda(t)|_2^2\\&
\leq \prod_{i=1}^N M(b_i)e^{\alpha(b_i)|\beta_i(t)|} |\nabla Y_\lambda(t)|_2\leq C\|x\|_1^2 \prod_{i=1}^N M(b_i)e^{\alpha(b_i)|\beta_i(t)|}.\end{aligned}$$So, integrating over $[0,T]$, we see that  $e^{\sum_{i=1}^N\beta_iB_i}Y_\lambda^\epsilon \in L^2([0,T];\ H_0^1(\mathcal{O}))$ $\mathbb{P}-$a.s.. Besides this, it is also a $\mathbb{P}-$a.s. continuous $(\mathcal{F}_t)-$adapted process in $L^2(\mathcal{O})$. We take in (\ref{e250}),  $\tilde{Z}= Y_\lambda^\epsilon$ and $$G=G_\lambda^\epsilon=-J_\epsilon(\text{div}\tilde{\psi}_\lambda(\nabla(e^{\sum_{i=1}^N\beta_i B_i} Y_\lambda))),$$so function $Y_\lambda^\epsilon$ satisfies (\ref{e250}). It yields by (\ref{e251}) that
\begin{equation}\label{e11e}\begin{aligned}&\frac{1}{2}|e^{\sum_{i=1}^N\beta_i(t) B_i}(Y^*(t)-Y_\lambda^\epsilon(t))|^2_2+\int_0^t\phi(e^{\sum_{i=1}^N\beta_i(s) B_i}Y^*(s))ds\\&
\leq \frac{1}{2}|x^*-x|_2^2+\int_0^t\phi(e^{\sum_{i=1}^N\beta_i(s) B_i}Y_\lambda^\epsilon(s))ds\\&
+\int_0^t\left<e^{\sum_{i=1}^N\beta_i(s) B_i}(Y^*(s)-Y_\lambda^\epsilon(s)),G_\lambda^\epsilon\right>ds.\end{aligned}\end{equation}
We estimate now the term $\left<e^{\sum_{i=1}^N\beta_i B_i}(Y^*-Y_\lambda^\epsilon),G_\lambda^\epsilon\right>$, by using the Green's formula, we get
$$\begin{aligned}&\left<e^{\sum_{i=1}^N\beta_i B_i}(Y^*-Y_\lambda^\epsilon),G_\lambda^\epsilon\right>
=\left<\nabla J_\epsilon (e^{\sum_{i=1}^N\beta_i B_i}Y^*)-\nabla (e^{\sum_{i=1}^N\beta_i B_i}Y_\lambda), \psi_\lambda(\nabla (e^{\sum_{i=1}^N\beta_i B_i}Y_\lambda))+\lambda\nabla(e^{\sum_{i=1}^N\beta_i B_i}Y_\lambda)\right>\\&
+\left<\zeta_\lambda^\epsilon,\psi_\lambda(\nabla (e^{\sum_{i=1}^N\beta_i B_i}Y_\lambda))+\lambda\nabla(e^{\sum_{i=1}^N\beta_i B_i}Y_\lambda)\right>,\end{aligned}$$where,
$$\zeta_\lambda^\epsilon=\nabla(e^{\sum_{i=1}^N\beta_i B_i}Y_\lambda)-\nabla J_\epsilon(e^{\sum_{i=1}^N\beta_i B_i}Y_\lambda^\epsilon).$$Since
$$\psi_\lambda(u)\cdot(u-v)\geq j_\lambda(u)-j_\lambda(v),\ \forall u,v \in \mathbb{R}^d,$$we deduce that
$$\begin{aligned}&\left<e^{\sum_{i=1}^N\beta_i B_i}(Y^*-Y_\lambda^\epsilon),G_\lambda^\epsilon\right>\leq \phi_\lambda(J_\epsilon(e^{\sum_{i=1}^N\beta_i B_i}Y^*))-\phi_\lambda(e^{\sum_{i=1}^N\beta_i B_i}Y_\lambda)-\lambda|\nabla(e^{\sum_{i=1}^N\beta_i B_i}Y_\lambda)|^2_2\\&
-\lambda\left<\Delta(e^{\sum_{i=1}^N\beta_i B_i}Y_\lambda), J_\epsilon(e^{\sum_{i=1}^N\beta_i B_i}Y^*)\right>+\left<\psi_\lambda(\nabla(e^{\sum_{i=1}^N\beta_i B_i}Y_\lambda))+\lambda\nabla(e^{\sum_{i=1}^N\beta_i B_i}Y_\lambda),\zeta_\lambda^\epsilon\right>,\end{aligned}$$where
$$\phi_\lambda(z)=\int_\mathcal{O}j_\lambda(\nabla z)d\xi,\ \forall z\in H_0^1(\mathcal{O}).$$Substituting this in (\ref{e11e}) we get that
\begin{equation}\label{e12e}\begin{aligned}&\frac{1}{2}|e^{\sum_{i=1}^N\beta_i(t) B_i}(Y^*(t)-Y_\lambda^\epsilon(t))|^2_2+\int_0^t\phi(e^{\sum_{i=1}^N\beta_i(s) B_i}Y^*(s))ds\\&
+\int_0^t\phi_\lambda(e^{\sum_{i=1}^N\beta_i(s) B_i}Y_\lambda(s))ds+\lambda\int_0^t|\nabla(e^{\sum_{i=1}^N\beta_i(s) B_i}Y_\lambda(s))|_2^2ds\\&
\leq \frac{1}{2}|x^*-x|_2^2+\int_0^t\phi(e^{\sum_{i=1}^N\beta_i(s) B_i}Y_\lambda^\epsilon(s))ds\\&
+\int_0^t\phi_\lambda(J_\epsilon(e^{\sum_{i=1}^N\beta_i(s) B_i}Y^*(s)))ds\\&
-\lambda\int_0^t\left<\Delta(e^{\sum_{i=1}^N\beta_i(s) B_i}Y_\lambda(s)), J_\epsilon(e^{\sum_{i=1}^N\beta_i B_i}Y^*(s))\right>ds\\&
+\int_0^t\left[\left<\psi_\lambda(\nabla(e^{\sum_{i=1}^N\beta_i(s) B_i}Y_\lambda(s)))+\lambda\nabla(e^{\sum_{i=1}^N\beta_i(s) B_i}Y_\lambda(s)),\zeta_\lambda^\epsilon(s)\right> \right]ds
\end{aligned}\end{equation}
Since,
$$|j_\lambda(\nabla u)-|\nabla u||\leq \frac{1}{2}\lambda,\ \forall u\in H_0^1(\mathcal{O}),$$we easily see that we have
\begin{equation}\label{e13e}|\phi(e^{\sum_{i=1}^N\beta_i(s) B_i}Y_\lambda(s))-\phi_\lambda(e^{\sum_{i=1}^N\beta_i(s) B_i}Y_\lambda(s))|\leq C\lambda,\ \forall s\in [0,T],\end{equation}and
\begin{equation}\label{e14e}\int_0^T|\phi_\lambda(J_\epsilon(e^{\sum_{i=1}^N\beta_i(s) B_i}Y^*(s)))-\phi(J_\epsilon(e^{\sum_{i=1}^N\beta_i(s) B_i}Y^*(s)))|ds\leq C\lambda.\end{equation}Using (\ref{e13e}) and (\ref{e14e}) in (\ref{e12e}), it yields
\begin{equation}\label{e15e}\begin{aligned}&\frac{1}{2}|e^{\sum_{i=1}^N\beta_i(t) B_i}(Y^*(t)-Y_\lambda^\epsilon(t))|^2_2+\int_0^t\phi(e^{\sum_{i=1}^N\beta_i(s) B_i}Y^*(s))ds+\lambda \int_0^t|\nabla e^{\sum_{i=1}^N\beta_i(s)B_i} Y_\lambda(s)|_2^2ds\\&
\leq \frac{1}{2}|x^*-x|_2^2+\int_0^t\phi(J_\epsilon(e^{\sum_{i=1}^N\beta_i(s) B_i}Y^*(s)))ds\\&
+\int_0^t\left[\phi(e^{\sum_{i=1}^N\beta_i(s) B_i}Y_\lambda^\epsilon(s))-\phi(e^{\sum_{i=1}^N\beta_i(s) B_i}Y_\lambda(s))\right]ds\\&
-\lambda\int_0^t\left<\Delta(e^{\sum_{i=1}^N\beta_i(s) B_i}Y_\lambda(s)), J_\epsilon(e^{\sum_{i=1}^N\beta_i(s) B_i}Y^*(s))\right>ds\\&
+C_{\lambda,\epsilon}\left(\int_0^t|\zeta_\lambda^\epsilon(s)|_2^2ds\right)^\frac{1}{2},\end{aligned}\end{equation}where
$$C_{\lambda,\epsilon}=\left(\left(\int_0^T|\tilde{\psi}_\lambda(\nabla e^{\sum_{i=1}^N\beta_i(s)B_i}Y_\lambda(s)|_2^2ds\right)\right)^\frac{1}{2}.$$

By \cite[Corrolary 8.1]{b3}, we know that
$$\int_0^t\phi(J_\epsilon(e^{\sum_{i=1}^N\beta_i(s) B_i}Y^*(s)))ds\leq \int_0^t\phi(e^{\sum_{i=1}^N\beta_i(s) B_i}Y^*(s))ds,\ \forall \epsilon>0,$$thus, letting $\epsilon\rightarrow 0$ in (\ref{e15e}) yields
\begin{equation}\label{e16e}\begin{aligned}&|e^{\sum_{i=1}^N\beta_i(t) B_i}(Y^*(t)-Y_\lambda(t))|^2_2\leq |x^*-x|_2^2-\lambda\mathbb{E}\int_0^t\left<\Delta(e^{\sum_{i=1}^N\beta_i(s) B_i}Y_\lambda(s)),e^{\sum_{i=1}^N\beta_i(s) B_i}Y^*(s)\right>ds,
\end{aligned}\end{equation}because : $\lim_{\epsilon\rightarrow 0}\int_0^T|\zeta_\lambda^\epsilon(s)|_2^2ds=0,\ \sup_{\epsilon\in(0,1)}C_{\lambda,\epsilon}<\infty,$ by Lemma \ref{lu1}  $e^{\sum_{i=1}^N\beta_i B_i}Y_\lambda\in L^2([0,T];\ H^2(\mathcal{O})) \ \  \mathbb{P}-\text{a.s.}$, and $e^{\sum_{i=1}^N\beta_i B_i}Y^*\in L^2([0,T];L^2(\mathcal{O}))\ \  \mathbb{P}-\text{a.s.}.$

By Lemma \ref{lu1}, by the commutativity of $\Delta$ with $e^{sB_i}$, for all $s\in\mathbb{R},\ i=1,...,N$, and (B7), we have that
$$\lim_{\lambda\rightarrow 0}\lambda\int_0^t\left<\Delta(e^{\sum_{i=1}^N\beta_i(s) B_i}Y_\lambda(s)),e^{\sum_{i=1}^N\beta_i B_i(s)}Y^*(s)\right>ds=\lim_{\lambda\rightarrow 0}\lambda\int_0^t\left<\Delta Y_\lambda(s)), Y^*(s)\right>ds=0.$$Hence, letting $\lambda\rightarrow 0$ in (\ref{e16e}), it follows  that
$$|X^*(t)-X(t)|_2^2\leq |x^*-x|_2^2,\ t\in[0,T] \ \ \ \mathbb{P}-\text{a.s.},$$ completing so the proof of the theorem, by letting $x\rightarrow x^*$ in $L^2(\mathcal{O}).$
\subsection{Positivity of solution}
We stress that physical models of non-linear diffusion are concerned with non-negative solutions to the equation (\ref{ioi20}). Hence, the next result is of most importance.
\begin{theorem}\label{Tit1} In Theorem \ref{t1} assume in addition that $x\geq0$, almost everywhere in $\mathcal{O}$. Then,
$$X(t,\xi)\geq 0 \ \text{almost everywhere in $(0,T)\times\mathcal{O}\times\Omega$.}$$
\end{theorem}
\begin{proof}
It is evident that it is enough to show that the solution $X_\lambda$ to (\ref{e}) is almost everywhere non-negative on $[0,T]\times\mathcal{O}\times\Omega$. To this end, by (B4) and the relation $X_\lambda=e^{\sum_{i=1}^N\beta_i B_i}Y_\lambda$, it suffices to show that the solution $Y_\lambda$ to (\ref{ee9}) stays non-negative. Let us denote by $Z_\lambda=-Y_\lambda$. Since $-\tilde{\psi}_\lambda(u)=\tilde{\psi}_\lambda(-u)$, it follows that $Z_\lambda$ satisfies
$$\frac{d}{dt}Z_\lambda(t)=e^{-\sum_{i=1}^N\beta_i B_i}\text{div}\tilde{\psi}_\lambda(\nabla(e^{\sum_{i=1}^N\beta_i B_i}Z_\lambda(t)))dt \text{ in }(0,T)\times\mathcal{O};\ Z_\lambda(0)=-x.$$Scalarly multiplying the above equation by $Z_\lambda^+$, yields, using again (B4)
$$\frac{1}{2}\frac{d}{dt}|Z_\lambda^+(t)|_2^2+\int_\mathcal{O}\tilde{\psi}_\lambda(\nabla(e^{\sum_{i=1}^N\beta_i B_i}Z^+_\lambda(t)))\cdot\nabla(e^{\sum_{i=1}^N\beta_i B_i}Z_\lambda^+(t))=0,$$where using the monotonicity of $\tilde{\psi}_\lambda$, we get
$$\frac{d}{dt}|Z_\lambda^+(t)|_2^2\leq 0,\ t\in[0,T],$$hence, $Z_\lambda^+(t)\equiv 0\ -$a.s., since $Z_\lambda^+(0)=(-x)^+=0.$ In consequence, $Y_\lambda^-(t)\equiv 0\ -$a.s, and the conclusion of the theorem follows immediately.
\end{proof}
\subsection{Finite time extinction  and further properties of the positive solution}Next, we are concerned with the problem of extinction in finite time of the solution, which is of fundamental nature for these kind of equations. We notice that, in the case of additive noise of the form $XdW$, this problem has been solved in \cite{gess}. Unfortunately, that result cannot be applied to our case, where the drift term contains space derivatives of the solution. However, we can obtain the following results.
\begin{theorem}\label{la17}Let $1\leq d\leq 2$. Let $X$ be as in Theorem \ref{t1}, with initial condition $x\in L^2(\mathcal{O})$; and let $\tau:=\inf\left\{t\geq 0:\ |X(t)|_2=0\right\}$. Then we have
\begin{equation}\label{wow9}\mathbb{P}[\tau <\infty]=1.\end{equation}
\end{theorem}
\begin{proof}Recall that, applying It\^o's formula in (\ref{e261}) to $|X_\lambda(t)|_2^2$, we have
\begin{equation}\label{wow2}d|X_\lambda(t)|_2^2+2\left<\tilde{\psi}_\lambda(\nabla X_\lambda(t)),\nabla X_\lambda(t)\right>=0,\ \forall t\geq0.\end{equation}Hence, for $\epsilon\in(0,1)$, we have, via (\ref{wow2}), that
\begin{equation}\label{wow3}(|X_\lambda(t)|_2^2+\epsilon)^\frac{1}{2}+\int_0^t (|X_\lambda(s)|_2^2+\epsilon)^{-\frac{1}{2}}\left<\tilde{\psi}_\lambda(\nabla X_\lambda(s)),\nabla X_\lambda(s)\right>ds=(|x|^2_2+\epsilon)^\frac{1}{2},\ \forall t\geq0.\end{equation}Recall that $\tilde{\psi}_\lambda (u)\cdot u\geq |u|-\lambda,\ \forall u\in \mathbb{R}^d$, we deduce that
\begin{equation}\label{wow4}\begin{aligned}\left<\tilde{\psi}_\lambda(\nabla X_\lambda(t)),\nabla X_\lambda(t)\right>&
\geq \int_\mathcal{O}|\nabla X_\lambda(t)|d\xi-\lambda \int_\mathcal{O}d\xi\\&
\geq \rho |X_\lambda(t)|_2-\lambda \int_\mathcal{O}d\xi, \end{aligned}\end{equation}
by the Sobolev embedding for $1\leq d\leq 2$
$$|\nabla y|_1\geq \rho |y|_{\frac{d}{d-1}},\ \forall y\in W^{1,1}_0(\mathcal{O}).$$So, plugging (\ref{wow4}) into (\ref{wow3}) we arrive to
\begin{equation}\label{wow5}(|X_\lambda(t)|_2^2+\epsilon)^\frac{1}{2}+\rho \int_0^t (|X_\lambda(s)|_2^2+\epsilon)^{-\frac{1}{2}} |X_\lambda(s)|_2ds-\lambda \int_\mathcal{O}d\xi\int_0^t(|X_\lambda(s)|_2^2+\epsilon)^{-\frac{1}{2}}ds\leq (|x|_2^2+\epsilon)^\frac{1}{2},\ \forall t\geq0.\end{equation}Taking expectation in (\ref{wow5}), we see that by (\ref{e8e}), Fatou's lemma and (\ref{e262}), we may let first $\lambda\rightarrow 0$, then let $\epsilon\rightarrow 0$, to get that
\begin{equation}\label{la1}\mathbb{E}|X(t)|_2+\rho \int_0^t\mathbb{P}[|X(s)|_2>0]ds\leq |x|_2,\ t>0,\end{equation}since
$$\int_0^t \mathbb{P}[|X(s)|_2>0]ds=\sup_{\epsilon>0}\int_0^t\mathbb{E}[|X(s)|_2(|X(s)_2+\epsilon)^{-1}]ds.$$Noticing that $\mathbb{P}[|X(s)|_2>0]=\mathbb{P}[\tau>s]$, it yields by (\ref{la1}) that
\begin{equation}\label{wow7}\mathbb{P}[\tau>t]\leq \frac{1}{\rho t}|x|_2,\end{equation}that immediately leads to (\ref{wow9}), as claimed.
\end{proof}
For the case $d=3$, we take $x\in L^{3}(\mathcal{O})$ such that $x\geq 0$; and obtain a similar result as in Theorem \ref{la17}, but for positive solutions. Firstly, let us show the next lemma.
\begin{lemma}\label{la2}Let $ d= 3$. For each $\lambda \in (0,1]$, let $X_\lambda$ be as in Proposition \ref{la3}, with initial condition $x\in L^{3}(\mathcal{O}),\ x\geq0$. Then, we have
\begin{equation}\label{la4}|X_\lambda(t)|_3^3+6\int_s^t\left(\int_\mathcal{O}X_\lambda(r) \tilde{\psi}_\lambda (\nabla X_\lambda(r)) \cdot \nabla X_\lambda(r) d\xi\right)dr=|X_\lambda(s)|_3^3 \ \mathbb{P}-\text{a.s.},\ \forall 0\leq s\leq t\leq T.\end{equation}In particular, it follows that
\begin{equation}\label{la31}| X_\lambda(t)|_3^3\leq |X_\lambda(s)|_3^3,\ \forall 0\leq s\leq t\leq T.\end{equation}\end{lemma}
\begin{proof}For $K\in\mathbb{N},\ K>\|x\|_1,$ let us define the $(\mathcal{F}_t)-$stopping time $$\theta_K:=\inf\left\{t\geq0:\ \|X_\lambda(t)\|_1>K\right\}.$$ By interpolation, we immediately see that
$$\mathbb{E}\int_0^{\theta_K}\|X_\lambda(s)\|_{1,3}^3ds\leq C\mathbb{E}\int_0^{\theta_K}\|X_\lambda(s)\|_{H^2(\mathcal{O})}^2|X_\lambda(s)|_3ds\leq CK\mathbb{E}\int_0^T\|X_\lambda(s)\|_{H^2(\mathcal{O})}^2ds<\infty,$$using the Sobolev embedding, in $d=3$,  $H^1(\mathcal{O})\subset L^3(\mathcal{O})$ and Proposition \ref{la3} (iii). Hence, we may apply Theorem 2.1 in \cite{krylov} for
$$\begin{aligned}& f_t:=\tilde{\psi}_\lambda(\nabla X_\lambda(t))(\leq 1+\lambda|\nabla X_\lambda(t)|)\\&
f_t^0:=\frac{1}{2}\sum_{i=1}^NB_i^2X_\lambda(t)\\&
g_t^i:=B_iX_\lambda(t)\\&
p=3,\end{aligned}$$keeping also in mind that the solution is positive, we get the following It\^o's  formula for the $L^3(\mathcal{O})-$norm $\mathbb{P}-$a.s.
\begin{equation}\label{la10}\begin{aligned}|X_\lambda(t\wedge \theta_K)|_3^3=&|X_\lambda(s\wedge\theta_K)|_3^3+3\sum_{i=1}^N\int_{s\wedge\theta_K}^{t\wedge\theta_K}\left(\int_\mathcal{O}X^2_\lambda(r)B_iX_\lambda(r)d\xi\right)d\beta_i(r)\\&
-6\int_{s\wedge\theta_K}^{t\wedge\theta_K}\left(\int_\mathcal{O}
X_\lambda(r)\tilde{\psi}_\lambda(\nabla X_\lambda(r))\cdot \nabla X_\lambda(r)d\xi\right)dr\\&
+3\sum_{i=1}^N\int_{s\wedge\theta_K}^{t\wedge\theta_K}\left(\frac{1}{2}\int_\mathcal{O}X_\lambda^2(r)B_i^2X_\lambda(r)+X_\lambda(r)|B_iX_\lambda(r)|^2d\xi\right)dr,\ 0\leq s\leq t\leq T.\end{aligned}\end{equation}
Taking advantage of the skew-adjointness of $B_i,\ i=1,...,N,$ simple computations show that $\int_{s\wedge\theta_K}^{t\wedge\theta_K}\int_\mathcal{O}X_\lambda^2B_iX_\lambda =0$ and $\int_{s\wedge\theta_K}^{t\wedge\theta_K}\int_\mathcal{O}\frac{1}{2}X_\lambda^2B_i^2X_\lambda+X_\lambda|B_iX_\lambda|^2=0$. This is indeed so. Firstly, we notice that, by the Holder inequality and the Sobolev embeddings, we have
$$\int_{s\wedge\theta_K}^{t\wedge\theta_K}\int_\mathcal{O}|\nabla(X_\lambda^2)|^2d\xi\leq C\int_{s\wedge\theta_K}^{t\wedge\theta_K}|X_\lambda|^2_4|\nabla X_\lambda|^2_4\leq C\int_{s\wedge\theta_K}^{t\wedge\theta_K}|\nabla X_\lambda|^2_2|\Delta X_\lambda|^2_2<CK^2\int_0^T|\Delta X_\lambda|^2_2<\infty$$ by  (\ref{mos}); and
$$\int_{s\wedge\theta_K}^{t\wedge\theta_K}\int_\mathcal{O}|\nabla(X_\lambda^3)|^2d\xi\leq C\int_{s\wedge\theta_K}^{t\wedge\theta_K}|X_\lambda|^4_6|\nabla X_\lambda|^2_6\leq C\int_{s\wedge\theta_K}^{t\wedge\theta_K}|\nabla X_\lambda|^4_2|\Delta X_\lambda|^2_2<CK^4\int_0^T|\Delta X_\lambda|^2_2<\infty$$ again by (\ref{mos}). Then,
$$\int_{s\wedge\theta_K}^{t\wedge\theta_K}\int_\mathcal{O}X_\lambda^2B_iX_\lambda =\frac{1}{3}\int_{s\wedge\theta_K}^{t\wedge\theta_K}\int_\mathcal{O}B_iX^3_\lambda =0,$$and
$$\int_{s\wedge\theta_K}^{t\wedge\theta_K}\int_\mathcal{O}X_\lambda^2B_i^2X_\lambda = -\int_{s\wedge\theta_K}^{t\wedge\theta_K}\int_\mathcal{O}B_iX_\lambda^2 B_iX_\lambda =-\int_{s\wedge\theta_K}^{t\wedge\theta_K}2\int_\mathcal{O}X_\lambda|B_iX_\lambda|^2.$$Consequently, by (\ref{la10}) and the monotonicity of $\tilde{\psi}_\lambda$, it follows that
$$|X_\lambda(t\wedge \theta_K)|_3^3\leq |x|_3^3,\ \forall t\in [0,T].$$Thus, we may let $K\rightarrow \infty$ in (\ref{la10}) to get relation (\ref{la4}), as wanted. Relation (\ref{la31}) is an easy consequence of (\ref{la4}) and the monotonicity of $\tilde{\psi}_\lambda$.
\end{proof}
Now we can state the finite time extinction result for the $d=3$ case.
\begin{theorem}\label{Tit3}Let $d=3$. Let $X$ be as in Theorem \ref{t1}, with initial condition $x\in L^3(\mathcal{O}),\ x\geq0$; and let $\tau:=\inf\left\{t\geq 0:\ |X(t)|_3=0\right\}$. Then we have
\begin{equation}\label{la18}\mathbb{P}[\tau <\infty]=1.\end{equation}
\end{theorem}
\begin{proof}First, let us notice that by (\ref{mos}), for $x\in H_0^1(\mathcal{O})$, it follows by Fatou's lemma that for some constant $C>0$, independent of $x$, we have
\begin{equation}\label{la11}\mathbb{E}\left[\sup_{t\in[0,T]}\|X(t)\|_1^2\right]\leq C\|x\|_1.\end{equation}Then, by interpolation, we get
$$\mathbb{E}[\sup_{t\in[0,T]}|X_\lambda(t)-X(t)|_3^2]\leq C\left(\mathbb{E}[\sup_{t\in[0,T]}|X_\lambda(t)-X(t)|_2^2]\right)^\frac{1}{2}\|x\|_1,$$from where, via (\ref{e8e}), we deduce that
\begin{equation}\label{lu13}\lim_{\lambda\rightarrow 0}\mathbb{E}\left[\sup_{t\in[0,T]}|X_\lambda(t)-X(t)|_3^2\right]=0.\end{equation}

Now, let $\epsilon\in(0,1)$. By (\ref{la4}), we get
\begin{equation}\label{la14}\sqrt[3]{|X_\lambda(t)|_3^3+\epsilon}+2\int_0^t\frac{1}{\sqrt[3]{(|X_\lambda(s)|^3_3+\epsilon)^2}}\int_\mathcal{O}X_\lambda(s)\tilde{\psi}_\lambda(\nabla X_\lambda(s))\cdot \nabla X_\lambda(s)d\xi ds=\sqrt[3]{|x|_3^3+\epsilon},\ t>0.\end{equation}

Notice that $\tilde{\psi}_\lambda(u)\cdot u\geq |u|-\lambda$. Hence, we have
$$\begin{aligned}&\int_\mathcal{O}2X_\lambda\tilde{\psi}_\lambda(\nabla X_\lambda)\cdot\nabla X_\lambda d\xi\geq \int_\mathcal{O}2X_\lambda (|\nabla X_\lambda|-\lambda)d\xi\\&
=\int_\mathcal{O}|\nabla(X_\lambda^ 2)|d\xi-2\lambda\int_\mathcal{O}X_\lambda d\xi\\&
\geq \rho|X_\lambda|_3^2-2\lambda|X_\lambda|_1,\end{aligned}$$
where we have used the embedding $\rho|y|_\frac{3}{2}\leq \|y\|_{1,1},\ \forall y\in W_0^{1,1}(\mathcal{O}).$ This plugged in (\ref{la14}) yields
\begin{equation}\label{la15}\begin{aligned}\sqrt[3]{|X_\lambda(t)|_3^3+\epsilon}+&6\rho\int_0^t\frac{1}{\sqrt[3]{(|X_\lambda(s)|^3_3+\epsilon)^2}}|X_\lambda(s)|_3^2 ds\\&
\leq \sqrt[3]{|x|_3^3+\epsilon}+12\lambda\int_0^t\frac{1}{\sqrt[3]{(|X_\lambda(s)|^3_3+\epsilon)^2}}|X_\lambda(s)|_1ds,\ t>0.\end{aligned}\end{equation}By (\ref{la31}) and (\ref{lu13}) we see that $|X(t)|_3$ is an $L^1-$ limit of supermartingales, hence itself a supermartingale. Then, making use of relation (\ref{lu13}), again, and arguing as in the proof of Theorem \ref{la17}, we let $\lambda\rightarrow 0$, then $\epsilon\rightarrow 0$ in (\ref{la15}), to get that $\mathbb{P}[\tau>t]\leq \frac{|x|_3}{6\rho t},$ which implies (\ref{la18}) as wanted.
\end{proof}
\subsection*{Acknowledgement.}Michael Rockner was supported by the DFG through the CRC 701. Ionu\c t Munteanu was supported by a post-doctoral fellowship of the Alexander von Humboldt Foundation.

\end{document}